\documentclass{amsart}

\usepackage{hyperref}
\usepackage{amsaddr}

\newtheorem{definition}{Definition}
\newtheorem{lemma}{Lemma}
\newtheorem{theorem}{Theorem}
\newtheorem{corollary}{Corollary}
\newtheorem{proposition}{Proposition}

\newtheorem*{mainResult*}{Main Result}

\newcommand{\sG}{\mathfrak{sG}} 
\newcommand{\FD}{\mathcal{FD}} 
\newcommand{\FDC}{\mathcal{FDC}} 
\newcommand{\R}{\mathbb{R}} 
\newcommand{\N}{\mathbb{N}} 
\newcommand{\Mk}{\mathbb{M}_k} 

\begin{document}

\title{Graph limits: An alternative approach to s-graphons}

\author{Martin Dole\v{z}al}

\thanks{The research was supported by the GA\v{C}R project 18-01472Y and RVO: 67985840.}

\address{Institute of Mathematics of the Czech Academy of Sciences \\
\v{Z}itn\'{a} 25, 115 67 Praha~1 \\
Czech Republic}

\maketitle

\begin{abstract}
We show that s-convergence of graph sequences is equivalent to the convergence of certain compact sets, called shapes, of Borel probability measures.
This result is analogous to the characterization of graphon convergence (with respect to the cut distance) by the convergence of envelopes, due to Dole\v{z}al, Greb\'{i}k, Hladk\'{y}, Rocha, and Rozho\v{n}.
\end{abstract}

\section{Introduction}

The recently developed theory of graph limits is an important tool in graph theory. Usually, a different approach is applied to dense graph sequences than to sparse graph sequences. For dense graph sequences, the convergence of subgraph densities ~\cite{LoSz2006,BCLSV2008} is employed in most cases.
For very sparse graph sequences, one may employ for example the Benjamini-Schramm convergence~\cite{BeSc2001} or the local-global convergence~\cite{BoRi2011,HaLoSz2014}.
As was noted in~\cite{KuLoSz2019}, in the case of each of the three convergence notions mentioned so far, the limit objects can be equivalently represented by symmetric Borel measures on the unit square. This simple observation was elaborated in~\cite{KuLoSz2019} to build up a completely new theory of graph limits, based on so called s-convergence, which can be applied to arbitrary graph sequences and whose limit objects are always nontrivial.
(Recall that other approaches to unifying the dense and the sparse graph limit theories can be found in~\cite{BaSz2018,NeOs20}.)
The key property of this new theory is compactness, in the sense that every graph sequence has an s-convergent subsequence. As hinted above, the limit objects, called s-graphons, are symmetric Borel probability measures on the unit square. A very encouraging fact is that, for dense graph sequences, there is not a big difference between graphons and s-graphons as limit objects. The only information which we lose by considering s-graphons instead of graphons is the edge density of the limit object (see~\cite[Chapter~11]{KuLoSz2019}).


Let us briefly recall the notion of s-convergence from~\cite{KuLoSz2019}. For every finite graph $G$ (resp. every s-graphon $\mu$) and every $k \in \N$, we consider a certain compact subset of $k$-by-$k$ matrices which is called the $k$-shape of $G$ (resp. $\mu$).
Now, a graph sequence $\{ G_n \}_{n=1}^{\infty}$ is said to be s-convergent to an s-graphon $\mu$ if, for every $k \in \N$, the $k$-shapes of $G_n$ are convergent to the $k$-shape of $\mu$ in the Hausdorff distance. Among the main results of~\cite{KuLoSz2019}, it is shown that every graph sequence has a subsequence which is s-convergent to some s-graphon, and that every s-graphon is a limit of some s-convergent graph sequence.

The $k$-shape of a graph consists of the density matrices corresponding to all fractional partitions of the vertex set of the graph into $k$ sets of the same size.
The fact that these partitions are fractional suggests that it may be a good idea to leave the discrete world and to approach s-convergence via continuity tools. This is exactly what we do in this paper. Our goal is to provide an alternative description of s-convergence which clearly shows the topological properties of this notion. To achieve this goal, we first assign to each s-graphon $\mu$ a new shape as a certain compact subset $C(\mu)$ of s-graphons.
Then we prove the following:

\begin{mainResult*}[cf. Theorem~\ref{thm:main}, Section~\ref{sec:result}]
For any sequence $\{ \mu_n \}_{n=1}^{\infty}$ of s-graphons, convergence of the $k$-shapes $C(\mu_n,k)$ for every $k\in\N$ is equivalent to the convergence of the shapes $C(\mu_n)$.
\end{mainResult*}

The corresponding result for graph sequences follows as well.
Thus we obtain a simple characterization of s-convergence using the hyperspace of compact subsets of s-graphons.

In~\cite{KuLoSz2019}, the $k$-shape of an s-graphon is defined with the help of certain $k$-tuples of non-negative Borel (or, equivalently, continuous) functions on the unit interval.
The subtle discomfort of this definition is that one usually needs to prove something for every $k \in \N$. Note that this can be easily bypassed by using our new notion of shapes instead of $k$-shapes. We define the shape of an s-graphon with the help of certain non-negative continuous (or, equivalently, bounded Borel) functions on the unit square; we call these functions fairly distributed.
So the main difference is that, instead of considering $k$-tuples of functions,
we consider only single (fairly distributed) functions.

Let us informally describe some connections of our work to~\cite{KuLoSz2019}.
Recall that the $k$-tuples of functions used in the definition of $k$-shapes have two basic properties:
(i) their sum is the constant function $1$, and (ii) the integral of each of them equals $1/k$.
Intuitively, fairly distributed functions have very similar properties, when considered as collections of their horizontal sections (which are functions on the unit interval):
(i) the integral of each vertical section equals $1$, and (ii) the integral of each horizontal section equals $1$.
Further, note that finite graphs, as well as matrices with non-negative values, can be naturally interpreted as measures (see also~\cite[Lemma~5.1]{KuLoSz2019}).
Using this interpretation, $k$-shapes can be understood as subsets of s-graphons.
As we explain in the proof Theorem~\ref{thm:main} (by showing the equivalence (1) $\Leftrightarrow$ (1')), it is irrelevant whether we consider the convergence of $k$-shapes in the hyperspace of compact subsets of matrices or in the hyperspace of compact subsets of s-graphons.
These observations already hint towards our main result.

The correspondence between our new approach to s-convergence and the original approach taken in~\cite{KuLoSz2019} is of a similar nature as the correspondence between Chapter~7 from~\cite{DGHRR} and one particular characterization of convergence of subgraph densities from~\cite{BCLSV2012}. The latter correspondence is explained in detail in~\cite[Chapter~7]{DGHRR}.
Recall that convergence of graphons, with respect to the cut distance, is characterized by Cauchyness of certain sets, called fractional $q$-quotients, for every $q \in \N$ (see~\cite[Theorem~3.5, (i) $\Leftrightarrow$ (iii)]{BCLSV2012}).
This Cauchyness is meant in the corresponding Hausdorff distance.
In~\cite[Proposition~7.2]{DGHRR} it is shown that this is further equivalent to the convergence of certain compact sets called envelopes.
So, instead of Cauchyness of all fractional $q$-quotients, it is enough to consider only the convergence of a single sequence of envelopes.
While shapes (defined in the current paper) are subsets of measures, envelopes are subsets of bounded Borel functions (which corresponds better to the graphon world).
Nevertheless, the space of envelopes and the space of shapes are endowed with very similar topologies.
Indeed, let us consider the weak* topology on the Banach space $L^{\infty}([0,1]^2)$ of bounded Borel functions (which is the dual space to $L^{1}([0,1]^2)$).
Then the hyperspace of weak*-compact subsets of $L^{\infty}([0,1]^2)$ can be equipped with the corresponding Vietoris topology, and envelopes are elements of this hyperspace.
Similarly, we may consider the Banach space of continuous functions on the unit square (with the supremum norm).
Then its dual Banach space can be identified with the space of all Radon measures on the unit square.
Indeed, a Radon measure $\mu$ acts on the space of continuous functions by
\[
\mu (f) = \int_{ (x,y) \in [0,1]^2 } f(x,y) \, d\mu(x,y), \qquad f \colon [0,1]^2 \rightarrow \R \textnormal{ continuous}.
\]
Now, observe that the weak topology on Borel probability measures is exactly the restriction of the weak* topology from the dual Banach space of Radon measures.
Therefore the main idea of this paper is the same as the main idea in~\cite{DGHRR}:
In both cases (convergence of subgraph densities/s-convergence), a certain weak*-compact subset (called envelope/shape) of a suitable dual Banach space is assigned to each graphon/s-graphon.
Then it is shown that convergence of graphons/s-graphons is equivalent to the convergence of the corresponding weak*-compact sets in the Vietoris topology.

Let us shortly summarize the content of this paper.
In Section~\ref{sec:preliminaries} we present some preliminaries and basic notations, including our new definition of shapes.
In Section~\ref{sec:comparison} we show that the shape of a given s-graphon can be easily recovered from its $k$-shapes and vice versa.
In Section~\ref{sec:result} we prove our main result, Theorem~\ref{thm:main}.
Recall that one of the main tools used in~\cite{KuLoSz2019} is a variant of Szemer\'{e}di's regularity lemma~\cite{Szemeredi} (see~\cite[Lemma~6.1]{KuLoSz2019}).
Similarly, we prove yet another variant of the regularity lemma in Section~\ref{sec:regularity} (Lemma~\ref{lem:regularity}) which is crucial for our main result.
In Section~\ref{sec:graphs} we apply our main result to graph sequences.
Note that some results in~\cite{KuLoSz2019} are formulated only for sequences of finite graphs (or matrices) instead of sequences of measures.
This seems to be merely for convenience as, in many cases, working with finite graphs (or matrices) simplifies the notational difficulties a little bit due to their more combinatorial nature.
In our case, we have already renounced working with discrete objects (as we define shapes as compact subsets of Borel probability measures).
Therefore it seems to be easier to formulate and prove our results for sequences of s-graphons.
In Section~\ref{sec:BorCont} we prove that it makes no difference whether we require the fairly distributed functions from the definition of shapes to be continuous or only bounded Borel, as bounded Borel fairly distributed functions produce a dense subset of the shape. This result is an analogy of~\cite[Lemma~3.1]{KuLoSz2019}.

Unfortunately, the results of this paper do not give a self-contained proof of the compactness of s-convergence.
Moreover, we rely on this compactness (proved by means of $k$-shapes in~\cite{KuLoSz2019}) in the proof of the implication (2) $\Rightarrow$ (1) in Theorem~\ref{thm:main}.
If we had a self-contained (independent of the results from~\cite{KuLoSz2019}) proof of the compactness given by the convergence of shapes then no reference to the compactness given by the convergence of $k$-shapes would be needed.
As the hyperspace of compact subsets of s-graphons is compact, for such a self-contained proof, it is enough to prove the following:
Whenever a sequence of shapes (of some s-graphons) is convergent to some compact subset $E$ of s-graphons then there is an s-graphon $\mu$ such that $E = C(\mu)$.

\section{Preliminaries}
\label{sec:preliminaries}

Recall that if $K$ is a compact metrizable topological space then the hyperspace of compact subsets of $K$ is a compact topological space when equipped with the Vietoris topology, and it is metrizable by the Hausdorff distance. While the Vietoris topology is well defined even without specifying any concrete compatible metric on $K$, the Hausdorff distance depends on the choice of the metric on $K$. Thus, when talking about convergence of compact subsets of $K$ in the hyperspace, it is sometimes convenient to use only the topological convergence (rather than the metric one), namely when the metric on $K$ is not clearly specified. For example, we state Theorem~\ref{thm:main} in terms of Vietoris topology as there is no canonical compatible metric on the compact space $\sG$ (but we pick one such metric in the proof of Lemma~\ref{lem:regularity} where we adjust it to our needs). For more information on the Vietoris topology see e.g.~\cite[Chapter~4.F]{Kechris}.

We denote by $\lambda$ the 1-dimensional Lebesgue measure restricted to the unit interval $[0,1]$, and by $\lambda^2$ the 2-dimensional Lebesgue measure restricted to the unit square $[0,1]^2$.

For every $k \in \N$, we denote by $\Mk$ the space of all real $k$-by-$k$ matrices equipped with the topology of $\R^{k \times k}$. We also define
\[
\Mk^* = \left\{ M \in \Mk \colon M \text{ is symmetric with non-negative values and } \sum_{i,j=1}^k M(i,j) = 1 \right\}.
\]

We say that a non-negative bounded Borel function $f \colon [0,1]^2 \rightarrow [0,\infty)$ is \emph{fairly distributed} (with respect to $\lambda$) if for every $x,y \in [0,1]$ it holds
\[
\int_{v\in[0,1]} f(x,v) \, d\lambda(v) = \int_{u\in[0,1]} f(u,y) \, d\lambda(u) = 1.
\]
Let $\FD$ denote the set of all non-negative bounded Borel fairly distributed (with respect to $\lambda$) functions on $[0,1]^2$. Let $\FDC$ denote the set of all functions in $\FD$ which are moreover continuous.

Let $\sG$ be the space of all s-graphons, that is, the space of all symmetric Borel probability measures on $[0,1]^2$. This space is equipped with the weak topology inherited from the space of all Borel probability measures on $[0,1]^2$ (the definition of this topology can be found e.g. in~\cite[Section~17.E]{Kechris}).
Recall that the weak topology on Borel probability measures on $[0,1]^2$ is compact (see e.g.~\cite[Theorem~17.22]{Kechris}).

For every $f \in \FD$ and every $\mu \in \sG$ we define a function $\varphi(f,\mu) \in L^1([0,1]^{2},\lambda^2)$ by
\[
\varphi(f,\mu)(u,v) = \int_{(x,y)\in[0,1]^2} f(x,u) f(y,v) \, d\mu(x,y), \qquad u,v \in [0,1].
\]
The fact that $\varphi(f,\mu)$ is $\lambda^2$-integrable immediately follows from the boundedness of $f$. We further define a Borel measure $\Phi(f,\mu)$ on $[0,1]^2$ such that it is absolutely continuous with respect to $\lambda^2$ with the Radon-Nikodym derivative equal to $\varphi(f,\mu)$, that is,
\[
\Phi(f,\mu)(A) = \int_{(u,v)\in A} \varphi(f,\mu)(u,v) \, d\lambda^2(u,v), \qquad A \subseteq [0,1]^2 \textnormal{ Borel}.
\]
By the fair distribution of $f$ we have
\begin{align*}
\Phi(f,\mu)([0,1]^2) &= \int_{(u,v)\in[0,1]^2} \int_{(x,y)\in[0,1]^2} f(x,u) f(y,v) \, d\mu(x,y) \, d\lambda^2(u,v) \\
&= \int_{(x,y)\in[0,1]^2} \int_{(u,v)\in [0,1]^2} f(x,u) f(y,v) \, d\lambda^2(u,v) \, d\mu(x,y) \\
&= \int_{(x,y)\in[0,1]^2} \left( \int_{u\in[0,1]} f(x,u) \, d\lambda(u) \int_{v\in[0,1]} f(y,v) \, d\lambda(v) \right) d\mu(x,y) \\
&= \int_{(x,y)\in[0,1]^2} 1 \, d\mu(x,y) = 1,
\end{align*}
and so $\Phi(f,\mu)$ is an s-graphon (the symmetry of $\Phi(f,\mu)$ follows from the obvious symmetry of $\varphi(f,\mu)$).

\begin{definition}
\label{def:shape}
For every $\mu \in \sG$, we define its \emph{shape} $C(\mu) \subseteq \sG$ by
\[
C(\mu) = \overline{\{ \Phi(f,\mu) \colon f \in \FDC \}},
\]
where the closure is taken in the weak topology.
\end{definition}

Later (in Proposition~\ref{prop:BC}) we will see that $C(\mu) = \overline{\{ \Phi(f,\mu) \colon f \in \FD \}}$, so the continuity of the functions $f \in\FDC$ in the above definition is irrelevant. As the set $\sG$ is clearly closed in the space of all Borel probability measures on $[0,1]^2$, it does not matter in the above definition whether we take the closure only in the space $\sG$ or in the bigger space of all Borel probability measures. Finally, as the weak topology on Borel probability measures on $[0,1]^2$ is compact, the shape $C(\mu)$ is a compact subset of $\sG$.

Recall that in \cite{KuLoSz2019}, for every s-graphon $\mu$ and every natural number $k$, the \emph{$k$-shape} $C(\mu,k)$ is defined in the following way: The $k$-shape $C(\mu,k)$ is the smallest closed set in $\Mk$ containing all matrices of the form
\[
\left( \int_{(x,y)\in[0,1]^2} f_i(x) f_j(y) \, d\mu(x,y) \right)_{i,j=1}^k,
\]
where $f_1, f_2, \ldots, f_k$ are non-negative Borel functions on $[0,1]$ such that their sum is the constant function $1$ and such that $\int_{x\in[0,1]} f_i(x) \, d\lambda(x) = 1/k$ for every $i=1,2,\ldots,k$. (In fact, the original definition from \cite{KuLoSz2019} deals with the Cantor set $\mathcal{C}$ instead of the interval $[0,1]$, and with measures on $\mathcal{C}^2$ and functions $f_1, f_2, \ldots, f_k$ defined on $\mathcal{C}$. However, these two approaches are equivalent as it is explained in detail in \cite[Chapter~8]{KuLoSz2019}.) By the first part of Lemma~3.1 from~\cite{KuLoSz2019}, we may further require the functions $f_1, f_2, \ldots, f_k$ to be even continuous instead of only Borel, and the smallest closed subset of $\Mk$ containing all the corresponding matrices still equals $C(\mu,k)$. (Again, Lemma~3.1 in~\cite{KuLoSz2019} is formulated for functions $f_1, f_2, \ldots, f_k$ defined on $\mathcal{C}$ instead of $[0,1]$ but the same proof clearly works for $[0,1]$ as well.)

For our purposes, we need to interpret elements of $C(\mu,k)$ as s-graphons. We do it by interpreting each $M \in \Mk$ with non-negative values as the unique measure $\mu_M$ on $[0,1]^2$ satisfying, for every $i,j=1,2,\ldots,k$, that $\mu_M$ is uniformly (with respect to $\lambda^2$) distributed on $[\frac{i-1}{k}, \frac{i}{k}] \times [\frac{j-1}{k}, \frac{j}{k}]$ with $\mu_M( [\frac{i-1}{k}, \frac{i}{k}] \times [\frac{j-1}{k}, \frac{j}{k}] ) = M(i,j)$. In other words, $\mu_M$ is an absolutely continuous (with respect to $\lambda^2$) measure with the Radon-Nikodym derivative equal to $k^2 M(i,j)$ at every point from $\left( \frac{i-1}{k}, \frac{i}{k} \right) \times \left( \frac{j-1}{k}, \frac{j}{k} \right)$, $i,j = 1,2,\ldots,k$. It is easy to check that if $M$ belongs to $\Mk^*$ (in particular, if $M \in C(\mu,k)$) then $\mu_M$ is an s-graphon. Now for every s-graphon $\mu$ and every natural number $k$, we can define
\[
\widetilde{C}(\mu,k) = \{ \mu_M \colon M \in C(\mu,k) \}.
\]
Note that the mapping $M \mapsto \mu_M$ from $\Mk^*$ to $\sG$ is continuous. In particular, this implies that the set $\widetilde{C}(\mu,k)$ is compact by the compactness of $C(\mu,k)$.

\section{Comparison of $k$-shapes and shapes}
\label{sec:comparison}

In Lemma~\ref{lem:shapes-a} and Lemma~\ref{lem:shapes-b}, we uncover the connection of the $k$-shapes $C(\mu,k)$ introduced in~\cite{KuLoSz2019} with the shape $C(\mu)$ from Definition~\ref{def:shape}.

\begin{lemma}
\label{lem:shapes-a}
For every s-graphon $\mu$, we have
\[
C(\mu) = \overline{ \bigcup_{k\in\N} \widetilde{C}(\mu,k) },
\]
where the closure is taken in the weak topology.
\end{lemma}

\begin{proof}
First, we prove the inclusion
\begin{equation}
\label{eq:inclusion1}
\overline{ \bigcup_{k\in\N} \widetilde{C}(\mu,k) } \subseteq C(\mu).
\end{equation}
Suppose that $k$ is a natural number, and let $f_1, f_2, \ldots, f_k$ be non-negative continuous functions on $[0,1]$ such that their sum is the constant function $1$ and such that $\int_{x\in[0,1]} f_i(x) \, d\lambda(x) = 1/k$ for every $i=1,2,\ldots,k$. Let $M \in C(\mu,k)$ be the matrix given by
\[
M(i,j) = \int_{(x,y)\in[0,1]^2} f_i(x) f_j(y) \, d\mu(x,y), \qquad i,j=1,2,\ldots,k,
\]
and let $\mu_M$ be the corresponding s-graphon from $\widetilde{C}(\mu,k)$. We define a function $g \colon [0,1]^2 \rightarrow [0,k]$ by
\[
g(x,y) = k f_i(x), \qquad x \in [0,1], y \in I_i, \qquad i=1,2,\ldots,k,
\]
where
\begin{equation}
\label{not:k-net}
\begin{aligned}
& I_i = \left[\frac{i-1}{k}, \frac{i}{k}\right) \qquad \textnormal{for } i < k, \\
& I_k = \left[\frac{k-1}{k}, 1\right].
\end{aligned}
\end{equation}
Then $g$ belongs to $\FD$. Moreover, for every $i,j=1,2,\ldots,k$ and every $(u,v) \in I_i \times I_j$ we have
\begin{align*}
\varphi(g,\mu)(u,v) &= \int_{(x,y)\in[0,1]^2} g(x,u) g(y,v) \, d\mu(x,y) \\
&= k^2 \int_{(x,y)\in[0,1]^2} f_i(x) f_j(y) \, d\mu(x,y) \\
&= k^2 M(i,j),
\end{align*}
and it follows that $\Phi(g,\mu) = \mu_M$. However, the function $g$ is not necessarily continuous on $[0,1]^2$, only its restrictions to each of the sets $[0,1] \times I_i$, $i=1,2,\ldots,k$, are continuous. To deal with this issue, we fix $\varepsilon \in (0, \frac{1}{2k})$. Then we define a new function $g_{\varepsilon} \colon [0,1]^2 \rightarrow [0,k]$ by
\[
g_{\varepsilon}(x,y) =
\begin{cases}
g(x,y) & \text{if } \left| y - \frac{i}{k} \right| > \varepsilon \text{ for every } i=1,2,\ldots,k-1, \\
\frac{1}{2\varepsilon} (\frac{i}{k} + \varepsilon - y) k f_i(x) \\
\quad + \frac{1}{2\varepsilon} (y - \frac{i}{k} + \varepsilon) k f_{i+1}(x) & \text{if } \left| y - \frac{i}{k} \right| \le \varepsilon,\quad i=1,2,\ldots,k-1.
\end{cases}
\]
It is straightforward to check that $g_{\varepsilon} \in \FDC$. Moreover, the values of $g$ and $g_{\varepsilon}$ may differ only on the set $[0,1] \times \bigcup_{i=1}^{k-1} \left[ \frac{i}{k} - \varepsilon, \frac{i}{k} + \varepsilon \right]$. Thus, if we denote by $\mu^{(1)}$ the marginal of $\mu$ on the first coordinate (which is, by the symmetry of $\mu$, the same as the marginal $\mu^{(2)}$ of $\mu$ on the second coordinate) then we have
\begin{align*}
\| g - g_{\varepsilon} \|_{L^1([0,1]^2, \mu^{(1)} \times \lambda)} &\le k (\mu^{(1)} \times \lambda) \left( [0,1] \times \bigcup_{i=1}^{k-1} \left[ \frac{i}{k} - \varepsilon, \frac{i}{k} + \varepsilon \right] \right) \\
&= 2k(k-1)\varepsilon.
\end{align*}
Consequently, it holds
\begin{equation*}
\begin{split}
&\| \varphi(g,\mu) - \varphi(g_{\varepsilon},\mu) \|_{L^1([0,1]^2, \lambda^2)} \\
\le & \int_{(u,v)\in[0,1]^2} \int_{(x,y)\in[0,1]^2} | g(x,u)g(y,v) - g_{\varepsilon}(x,u)g_{\varepsilon}(y,v) | \, d\mu(x,y) \, d\lambda^2(u,v) \\
\le & \int_{(u,v)\in[0,1]^2} \int_{(x,y)\in[0,1]^2} g(x,u) | g(y,v) - g_{\varepsilon}(y,v) | \, d\mu(x,y) \, d\lambda^2(u,v) \\
& \qquad + \int_{(u,v)\in[0,1]^2} \int_{(x,y)\in[0,1]^2} g_{\varepsilon}(y,v) | g(x,u) - g_{\varepsilon}(x,u) | \, d\mu(x,y) \, d\lambda^2(u,v) \\
\le & k \int_{v\in[0,1]} \int_{(x,y)\in[0,1]^2} | g(y,v) - g_{\varepsilon}(y,v) | \, d\mu(x,y) \, d\lambda(v) \\
& \qquad + k \int_{u\in[0,1]} \int_{(x,y)\in[0,1]^2} | g(x,u) - g_{\varepsilon}(x,u) | \, d\mu(x,y) \, d\lambda(u) \\
= & k \| g - g_{\varepsilon} \|_{L^1([0,1]^2,\mu^{(2)}\times\lambda)} + k \| g - g_{\varepsilon} \|_{L^1([0,1]^2,\mu^{(1)}\times\lambda)} \\
\le & 4k^2(k-1)\varepsilon.
\end{split}
\end{equation*}
It follows that the measures $\Phi(g_{\varepsilon},\mu) \in C(\mu)$ weakly converge to $\Phi(g,\mu) = \mu_M$ as $\varepsilon$ tends to $0$. Thus $\mu_M$ belongs to the closed set $C(\mu)$. By the continuity of the mapping $M \mapsto \mu_M$ from $\Mk^*$ to $\sG$ it follows that, for every $k \in \N$, the set $C(\mu)$ contains a dense subset of $\widetilde{C}(\mu,k)$. As the shape $C(\mu)$ is closed, it follows that $\bigcup_{k\in\N} \widetilde{C}(\mu,k) \subseteq C(\mu)$, and inclusion~\eqref{eq:inclusion1} follows.

Now we show the other inclusion. It is enough to prove that $\Phi(f,\mu) \in \overline{ \bigcup_{k\in\N} \widetilde{C}(\mu,k) }$ for every $f \in \FDC$ as the measures $\Phi(f,\mu)$, $f \in \FDC$, are dense in $C(\mu)$. So let us fix a function $f \in \FDC$. We define $C = \sup_{(x,y)\in[0,1]^2}f(x,y)$, and we pick $\varepsilon > 0$. By the uniform continuity of $f$, there is $k \in \N$ such that $| f(x,y) - f(x',y') | < \varepsilon$ whenever the points $(x,y), (x',y')$ from $[0,1]^2$ are such that $|x-x'| \le \frac{1}{k}$ and $|y-y'| \le \frac{1}{k}$. Let $I_i$, $i=1,2,\ldots,k$, be the intervals given by~\eqref{not:k-net}. Let $f_1,f_2,\ldots,f_k \colon [0,1] \rightarrow [0, \frac{C}{k}]$ be functions given by
\[
f_i(x) = k \int_{(u,v) \in I_j \times I_i}f(u,v) \, d\lambda(u,v), \qquad x \in I_j, \qquad i,j = 1,2,\ldots,k.
\] 
It is easy to verify that $f_1, f_2, \ldots, f_k$ are non-negative Borel functions on $[0,1]$ such that their sum is the constant function $1$ and such that $\int_{x\in[0,1]} f_i(x) \, d\lambda(x) = 1/k$ for every $i=1,2,\ldots,k$. Thus the matrix $M$ given by
\[
M(i,j) = \int_{(x,y)\in[0,1]^2} f_i(x) f_j(y) \, d\mu(x,y), \qquad i,j=1,2,\ldots,k,
\]
belongs to the $k$-shape $C(\mu,k)$. For every $i,j = 1,2,\ldots,k$ and every $(u,v) \in I_i \times I_j$, it holds
\begin{equation*}
\begin{split}
& \left| k^2M(i,j) - \varphi(f,\mu)(u,v) \right| \\
= & \left| k^2 \int_{(x,y)\in[0,1]^2} f_i(x) f_j(y) \, d\mu(x,y) - \int_{(x,y)\in[0,1]^2} f(x,u)f(y,v) \, d\mu(x,y) \right| \\
\le & \sum_{s,t=1}^{k} \int_{(x,y) \in I_s \times I_t} \left| k^2 f_i(x) f_j(y) - f(x,u) f(y,v) \right| \, d\mu(x,y) \\
\le & \sum_{s,t=1}^{k} \int_{(x,y) \in I_s \times I_t} \left| k^2 f_i(x) f_j(y) - k f_i(x) f(y,v) \right| \, d\mu(x,y) \\
& \qquad + \sum_{s,t=1}^{k} \int_{(x,y) \in I_s \times I_t} \left| k f_i(x) f(y,v) - f(x,u) f(y,v) \right| \, d\mu(x,y) \\
\le & C \sum_{s,t=1}^{k} \int_{(x,y) \in I_s \times I_t} \left| k f_j(y) - f(y,v) \right| \, d\mu(x,y) \\
& \qquad + C \sum_{s,t=1}^{k} \int_{(x,y) \in I_s \times I_t} \left| k f_i(x) - f(x,u) \right| \, d\mu(x,y).
\end{split}
\end{equation*}
In the last two integrands, the expression $k f_j(y)$ is equal to the mean value of $f$ on $I_t \times I_j$, and so it is $\varepsilon$-close to $f(y,v)$ for every $(y,v) \in I_t \times I_j$; similarly $k f_i(x)$ is $\varepsilon$-close to $f(x,u)$ for every $(x,u) \in I_s \times I_i$.
We conclude that
\[
\left| k^2M(i,j) - \varphi(f,\mu)(u,v) \right| \le 2C \sum_{s,t=1}^k \varepsilon \mu(I_s \times I_t) = 2C\varepsilon.
\]
It follows that, by the choice of $\varepsilon > 0$, the Radon-Nikodym derivative (with respect to $\lambda^2$) of $\mu_M \in \bigcup_{k \in \N} \widetilde{C}(\mu,k)$ can be made arbitrarily close to $\varphi(f,\mu)$ in the space $L^1([0,1]^2,\lambda^2)$. Consequently, the measure $\Phi(f,\mu)$ is in the closure of the set $\bigcup_{k \in \N} \widetilde{C}(\mu,k)$. This completes the proof.
\end{proof}

\begin{lemma}
\label{lem:shapes-b}
For every s-graphon $\mu$ and every $k \in \N$ it holds
\[
\widetilde{C}(\mu,k) = C(\mu) \cap \{ \mu_M \colon M \in \Mk^* \}.
\]
\end{lemma}

\begin{proof}
The inclusion $\subseteq$ is an immediate consequence of Lemma~\ref{lem:shapes-a}.

To prove the opposite inclusion, we fix $M \in \Mk^*$ such that $\mu_M \in C(\mu)$.
We need to show that $M \in C(\mu,k)$. Let $I_i$, $i=1,2,\ldots,k$, be the intervals introduced in the proof of Lemma~\ref{lem:shapes-a} by~\eqref{not:k-net}. Recall that one of the many known characterizations of the weak convergence states that a sequence $\{ \nu_n \}_{n=1}^{\infty}$ of Borel probability measures on $[0,1]^2$ is weakly convergent to $\mu_M$ if and only if $\lim_{n\rightarrow\infty} \nu_n(A) = \mu_M(A)$ for every Borel set $A \subseteq [0,1]^2$ whose boundary (i.e. the set of all points in the closure of $A$ which are not interior points of $A$) is of $\mu_M$-measure zero (see e.g.~\cite[Theorem~17.20]{Kechris}). As $\mu_M \in C(\mu) = \overline{\{ \Phi(f,\mu) \colon f \in \FDC \}}$, this characterization implies that for any $\varepsilon > 0$ there is $f \in \FDC$ such that
\begin{equation}
\label{eq:close1}
\left| \Phi(f,\mu)( I_i \times I_j ) - \mu_M( I_i \times I_j ) \right| < \varepsilon, \qquad i,j=1,2,\ldots,k.
\end{equation}
Now we define functions $f_1,f_2,\ldots,f_k \colon [0,1] \rightarrow \R$ by
\[
f_i(x) = \int_{u \in I_i} f(x,u) \, d\lambda(u), \qquad x \in [0,1], \qquad i=1,2,\ldots,k.
\]
It is easy to verify that $f_1,f_2,\ldots,f_k$ are non-negative Borel (even continuous) functions such that their sum is the constant function $1$ and such that $\int_{x\in[0,1]} f_i(x) \, d\lambda(x) = 1/k$ for every $i=1,2,\ldots,k$. The corresponding matrix belonging to the $k$-shape $C(\mu,k)$ has the following entry on the position $(i,j)$:
\begin{equation}
\label{eq:close2}
\begin{split}
& \int_{(x,y)\in[0,1]^2} f_i(x) f_j(y) \, d\mu(x,y) \\
=& \int_{(x,y) \in [0,1]^2} \left( \int_{u \in I_i} f(x,u) \, d\lambda(u) \int_{v \in I_j} f(y,v) \, d\lambda(v) \right) \, d\mu(x,y) \\
=& \int_{(u,v) \in I_i \times I_j}
\int_{(x,y)\in [0,1]^2} f(x,u) f(y,v) \, d\mu(x,y) \, d\lambda^2(u,v) \\
=& \int_{(u,v) \in I_i \times I_j}
\varphi(f,\mu)(u,v) \, d\lambda^2(u,v) = \Phi(f,\mu) (I_i \times I_j).
\end{split}
\end{equation}
Combining~\eqref{eq:close1} and~\eqref{eq:close2} together shows that the matrix $M$ can be approximated in $\Mk$ with an arbitrary precision by a matrix belonging to $C(\mu,k)$. As the $k$-shape $C(\mu,k)$ is a closed subset of $\Mk$, it follows that $M \in C(\mu,k)$ as we wanted.
\end{proof}

Following the notation from~\cite[Chapter~11]{KuLoSz2019}, let $\mathcal{X}_s$ be the set of isomorphism classes of s-graphons, where two s-graphons $\nu_1, \nu_2$ are isomorphic if $C(\nu_1,k) = C(\nu_2,k)$ for every $k \in \N$. By Lemma~\ref{lem:shapes-a} and Lemma~\ref{lem:shapes-b} we have the following immediate corollary.

\begin{corollary}
\label{cor:shapes}
Two s-graphons $\nu_1, \nu_2$ are isomorphic if and only if $C(\nu_1) = C(\nu_2)$.
\end{corollary}

\begin{proof}
If $\nu_1$ and $\nu_2$ are isomorphic then their shapes are the same by Lemma~\ref{lem:shapes-a}. If, on the other hand, $\nu_1$ and $\nu_2$ have the same shapes then they are isomorphic by Lemma~\ref{lem:shapes-b}.
\end{proof}

\section{Regularity lemma}
\label{sec:regularity}

The following lemma is one of the key steps to prove Theorem~\ref{thm:main}. Note that it is a certain quantitative extension of the inclusion $C(\mu) \subseteq \overline{ \bigcup_{k \in \N} \widetilde{C} (\mu,k) }$ from Lemma~\ref{lem:shapes-a}.

\begin{lemma}
\label{lem:regularity}
Let $\rho$ be an arbitrary metric on $\sG$ compatible with the weak topology. Then for every $\varepsilon > 0$ there is $K \in \N$ such that for every $\nu \in \sG$ it holds
\begin{equation}
\label{eq:net}
d_{H}^{\rho} \left( C(\nu), \widetilde{C}(\nu,K) \right) \le \varepsilon,
\end{equation}
where $d_{H}^{\rho}$ is the Hausdorff distance on the hyperspace of compact subsets of $\sG$ which is obtained from the metric $\rho$.
\end{lemma}

\begin{proof}
As the hyperspace of compact subsets of $\sG$ is compact, any two compatible metrics $d_1, d_2$ on the hyperspace are uniformly equivalent in the sense that for every $\varepsilon > 0$ there is $\delta > 0$ such that
\begin{multline*}
d_1 (E,F) \le \delta \Rightarrow d_2 (E,F) \le \varepsilon \\
\textnormal{and } d_2 (E,F) \le \delta \Rightarrow d_1 (E,F) \le \varepsilon, \qquad E,F \subseteq \sG \textnormal{ compact}.
\end{multline*}
In particular, every two Hausdorff distances on the hyperspace are uniformly equivalent. Therefore it is enough to prove the statement only for one fixed metric $\rho$ on $\sG$ (and for the corresponding Hausdorff distance $d_{H}^{\rho}$ on the hyperspace). It is easy to check that one possible choice of the compatible metric on $\sG$ is the metric $\rho$ given by
\begin{equation}
\label{eq:metric}
\rho(\nu_1,\nu_2) = \sum_{j=1}^{\infty} \frac{1}{2^j} \left| \int_{ (x,y) \in [0,1]^2 } h_j(x,y) \, d\nu_1(x,y) - \int_{ (x,y) \in [0,1]^2 } h_j(x,y) \, d\nu_2(x,y) \right|,
\end{equation}
where $\{ h_j \colon j \in \N \}$ is a fixed countable dense subset of continuous functions on $[0,1]^2$ with values in $[0,1]$. This is our choice of the compatible metric on $\sG$.

We fix $\varepsilon > 0$. We find $J \in \N$ such that
\[
\sum_{j=J+1}^{\infty} \frac{1}{2^j} < \frac{\varepsilon}{2}.\]
Then we find $K \in N$ such that $\left| h_j(x,y) - h_j(x',y') \right| < \frac{\varepsilon}{2}$, $j=1,2,\ldots,J$, whenever the points $(x,y), (x',y')$ from $[0,1]^2$ are such that $|x-x'| \le \frac{1}{K}$ and $|y-y'| \le \frac{1}{K}$. Now let $\nu$ be an arbitrary s-graphon, we need to check~\eqref{eq:net}. By Lemma~\ref{lem:shapes-a} we have
\[
\widetilde{C}(\nu,K) \subseteq \overline{ \bigcup_{k\in\N} \widetilde{C}(\nu,k) } = C(\nu),
\]
and so we only need to show that for every $k \in \N$ and for every $M \in C(\nu,k)$ there is $N \in C(\nu,K)$ such that $\rho(\mu_M,\mu_N) \le \varepsilon$. So let $f_1, f_2, \ldots, f_k$ be non-negative Borel functions on $[0,1]$ such that their sum is the constant function $1$ and such that $\int_{x\in[0,1]} f_i(x) \, d\lambda(x) = 1/k$ for every $i=1,2,\ldots,k$, and let $M \in C(\nu,k)$ be the matrix given by
\[
M(i,j) = \int_{(x,y)\in[0,1]^2} f_i(x) f_j(y) \, d\nu(x,y), \qquad i,j=1,2,\ldots,k.
\]
For every $r=1,2,\ldots,K$ we define a non-negative Borel function $g_r$ on $[0,1]$ by
\[
g_r = k \sum_{i=1}^k \lambda \left( \left( \frac{i-1}{k}, \frac{i}{k} \right) \cap \left( \frac{r-1}{K}, \frac{r}{K} \right) \right) f_i.
\]
Then we have
\[
\sum_{r=1}^K g_r = \sum_{i=1}^k k \sum_{r=1}^K \lambda \left( \left( \frac{i-1}{k}, \frac{i}{k} \right) \cap \left( \frac{r-1}{K}, \frac{r}{K} \right) \right) f_i = \sum_{i=1}^k f_i \equiv 1,
\]
and (for each $r=1,2,\ldots,K$) also
\begin{multline*}
\int_{x \in [0,1]} g_r(x) \, d\lambda(x) = \sum_{i=1}^k \lambda \left( \left( \frac{i-1}{k}, \frac{i}{k} \right) \cap \left( \frac{r-1}{K}, \frac{r}{K} \right) \right) k \int_{x \in [0,1]} f_i(x) \, d\lambda(x) \\
= \sum_{i=1}^k \lambda \left( \left( \frac{i-1}{k}, \frac{i}{k} \right) \cap \left( \frac{r-1}{K}, \frac{r}{K} \right) \right) = \frac{1}{K}.
\end{multline*}
It follows that the matrix $N$ given by
\[
N(r,s) = \int_{(x,y)\in[0,1]^2} g_r(x) g_s(y) \, d\nu(x,y), \qquad r,s=1,2,\ldots,K,
\]
belongs to the $K$-shape $C(\nu,K)$. Moreover, for every $r,s=1,2,\ldots,K$ it holds
\begin{equation}
\label{eq:measures-comparison}
\begin{split}
& \mu_N \left( \left( \frac{r-1}{K}, \frac{r}{K} \right) \times \left( \frac{s-1}{K}, \frac{s}{K} \right) \right) = \int_{(x,y)\in[0,1]^2} g_r(x) g_s(y) \, d\nu(x,y) \\
=& \sum_{i,j=1}^k \lambda \left( \left( \frac{i-1}{k}, \frac{i}{k} \right) \cap \left( \frac{r-1}{K}, \frac{r}{K} \right) \right) \lambda \left( \left( \frac{j-1}{k}, \frac{j}{k} \right) \cap \left( \frac{s-1}{K}, \frac{s}{K} \right) \right) k^2M(i,j) \\
=& \sum_{i,j=1}^k \mu_M \left( \left( \left( \frac{i-1}{k}, \frac{i}{k} \right) \cap \left( \frac{r-1}{K}, \frac{r}{K} \right) \right) \times \left( \left( \frac{j-1}{k}, \frac{j}{k} \right) \cap \left( \frac{s-1}{K}, \frac{s}{K} \right) \right) \right) \\
=& \mu_M \left( \left( \frac{r-1}{K}, \frac{r}{K} \right) \times \left( \frac{s-1}{K}, \frac{s}{K} \right) \right).
\end{split}
\end{equation}
By~\eqref{eq:measures-comparison} and by the choice of $K$, we easily conclude that
\[
\left| \int_{ (x,y) \in [0,1]^2 } h_j(x,y) \, d\mu_M(x,y) - \int_{ (x,y) \in [0,1]^2 } h_j(x,y) \, d\mu_N(x,y) \right| < \frac{\varepsilon}{2}
\]
for every $j=1,2,\ldots,J$. Consequently, it holds that
\begin{multline*}
\rho(\mu_M,\mu_N) = \sum_{j=1}^{\infty} \frac{1}{2^j} \left| \int_{ (x,y) \in [0,1]^2 } h_j(x,y) \, d\mu_M(x,y) - \int_{ (x,y) \in [0,1]^2 } h_j(x,y) \, d\mu_N(x,y) \right| \\
< \sum_{j=1}^{J} \frac{\varepsilon}{2^{j+1}} + \sum_{j=J+1}^{\infty} \frac{1}{2^j} < \frac{\varepsilon}{2} + \frac{\varepsilon}{2} = \varepsilon,
\end{multline*}
which completes the proof.
\end{proof}

\section{Main result}
\label{sec:result}

Now we are ready to prove our main result.

\begin{theorem}
\label{thm:main}
Let $\mu$ and $\mu_n$, $n \in \N$, be s-graphons. Then the following conditions are equivalent:
\begin{itemize}
\item[(1)] $\forall_{k \in \N} \colon \lim_{n\rightarrow\infty} C(\mu_n,k) = C(\mu,k)$ in the Vietoris topology on the hyperspace of compact subsets of $\Mk$,
\item[(2)] $\lim_{n\rightarrow\infty} C(\mu_n) = C(\mu)$ in the Vietoris topology on the hyperspace of compact subsets of $\sG$.
\end{itemize}
\end{theorem}

\begin{proof}
The mapping $M \mapsto \mu_M$ from $\Mk^*$ to $\sG$ is one-to-one and continuous. As the space $\Mk^*$ is compact, it follows that the mapping $M \mapsto \mu_M$ is even a homeomorphism of $\Mk^*$ onto its image. This implies that the mapping $F \mapsto \left\{ \mu_M \colon M \in F \right\}$ is a homeomorphism of the hyperspace of compact subsets of $\Mk^*$ onto its image (which is a subset of the hyperspace of compact subsets of $\sG$). As a consequence, condition~(1) is equivalent to the following condition:
\begin{itemize}
\item[(1')] $\forall_{k \in \N} \colon \lim_{n\rightarrow\infty} \widetilde{C}(\mu_n,k) = \widetilde{C}(\mu,k)$ in the Vietoris topology on the hyperspace of compact subsets of $\sG$.
\end{itemize}

Now suppose that condition~(1') holds. We fix some metric $\rho$ on $\sG$ which is compatible with the weak topology, and we denote by $d_{H}^{\rho}$ the Hausdorff distance on the hyperspace of compact subsets of $\sG$ which is obtained from the metric $\rho$. Fix $\varepsilon > 0$. By Lemma~\ref{lem:regularity}, there is $K \in \N$ such that for every s-graphon $\nu$ we have~\eqref{eq:net}. Let $N \in \N$ be such that for every $n \ge N$ it holds $d_{H}^{\rho} ( \widetilde{C}(\mu_n,K), \widetilde{C}(\mu,K) ) < \varepsilon$. Then for every $n \ge N$ we have
\begin{equation*}
\begin{split}
& d_{H}^{\rho} ( C(\mu_n), C(\mu) ) \\
\le & d_{H}^{\rho} ( C(\mu_n), \widetilde{C}(\mu_n,K) ) + d_{H}^{\rho} ( \widetilde{C}(\mu_n,K), \widetilde{C}(\mu,K) ) + d_{H}^{\rho} ( \widetilde{C}(\mu,K), C(\mu) ) < 3\varepsilon,
\end{split}
\end{equation*}
and so condition~(2) holds. So we have proved (1) $\Leftrightarrow$ (1') $\Rightarrow$ (2).

For $\nu \in \sG$ we denote by $[\nu]$ the isomorphism class of $\nu$ (which is an element of the space $\mathcal{X}_s$ defined before Corollary~\ref{cor:shapes}). Recall that by Corollary~\ref{cor:shapes}, two s-graphons $\nu_1, \nu_2$ are isomorphic if and only if $C(\nu_1) = C(\nu_2)$. Therefore the following two topologies on the space $\mathcal{X}_s$ are well defined:
\begin{itemize}
\item[(A)] the coarsest topology such that the mapping $[\nu] \mapsto C(\nu,k)$ from $\mathcal{X}_s$ to the hyperspace of compact subsets of $\Mk^*$ is continuous for every $k \in \N$,
\item[(B)] the coarsest topology such that the mapping $[\nu] \mapsto C(\nu)$ from $\mathcal{X}_s$ to the hyperspace of compact subsets of $\sG$ is continuous.
\end{itemize}

The already proved implication (1) $\Rightarrow$ (2) shows that the identity mapping on $\mathcal{X}_s$ is continuous from the topology given by (A) to the topology given by (B). It only remains to observe that the topology given by (A) is compact, as then the two topologies coincide and the equivalence of conditions~(1) and~(2) follows.

The compactness was already essentially proved in~\cite{KuLoSz2019} but let us explain it in detail. Let $\prod_{k\in\N}\mathcal{K}(\Mk^*)$ be the product space of the hyperspaces of compact subsets of $\Mk^*$. Note that this space is compact metrizable, so we can fix a compatible metric $\sigma$ on $\prod_{k\in\N}\mathcal{K}(\Mk^*)$. Now let $\{\nu_n\}_{n=1}^\infty$ be an arbitrary sequence of s-graphons. By~\cite[Theorem~4.7]{KuLoSz2019} we can find, for every $n\in\N$, a graph sequence $\{ G_n^i \}_{i=1}^\infty$ which is s-convergent and its limit is $\mu_n$. Let $\mathcal S_n^i$ denote the element of $\prod_{k\in\N}\mathcal{K}(\Mk^*)$ which has the $k$-shape of $G_n^i$ on its $k$th coordinate. Similarly, let $\mathcal M_n$ denote the element of $\prod_{k\in\N}\mathcal{K}(\Mk^*)$ which has the $k$-shape $C(\nu_n,k)$ on its $k$th coordinate. Then, for every $n\in\N$, we can easily find $i_n\in\N$ such that
\begin{equation}
\label{eq:skoro}
\sigma(\mathcal S_n^{i_n},\mathcal M_n)<\frac 1n.
\end{equation}
By passing to a subsequence, we may assume that the graph sequence $\{ G_n^{i_n} \}_{n=1}^\infty$ is s-convergent. By~\cite[Theorem~4.5]{KuLoSz2019} there is an s-graphon $\nu$ such that for every $k\in\N$, the $k$-shapes of $G_n^{i_n}$ converge to $C(\nu,k)$ (when $n$ goes to infinity). Finally, by~(\ref{eq:skoro}) it follows that the $k$-shapes $C(\nu_n,k)$ converge to $C(\nu,k)$ as well. This proves the compactness of the topology given by (A).
\end{proof}

\section{Convergence of graph sequences}
\label{sec:graphs}

Recall that to each matrix $M \in \Mk^*$ we associated an s-graphon $\mu_M$. Now, if $G$ is an arbitrary finite graph (with a non-empty edge set) then we can consider its adjacency matrix $A_G$. After normalizing $A_G$ by its $l^1$-norm we obtain the matrix $\widetilde{A}_G = \frac{1}{\| A_G \|_{l^1}} A_G$ (which belongs to $\Mk^*$ where $k$ is the number of vertices of $G$) and the corresponding s-graphon $\mu_{\widetilde{A}_G}$. Therefore we can define the shape of every finite graph $G$ (with a non-empty edge set) as the shape $C(\mu_{\widetilde{A}_G})$. Following~\cite{KuLoSz2019}, we say that a graph sequence $\{ G_n \}_{n=1}^{\infty}$ is s-convergent (to an s-graphon $\mu$) if and only if the $k$-shapes of $\mu_{\widetilde{A}_{G_n}}$ are convergent (to the $k$-shape of $\mu$) for every $k \in \N$. Thus we immediately obtain the following corollary of Theorem~\ref{thm:main} (we tacitly assume that each of the graphs $G_n$, $n \in \N$, has a non-empty edge set).

\begin{corollary}
\label{cor:graphs}
Let $\{ G_n \}_{n=1}^{\infty}$ be a graph sequence and let $\mu$ be an s-graphon. Then the following conditions are equivalent:
\begin{itemize}
\item[(1)] $\{ G_n \}_{n=1}^{\infty}$ is s-convergent to $\mu$,
\item[(2)] $\lim_{n\rightarrow\infty} C(G_n) = C(\mu)$ in the Vietoris topology on the hyperspace of compact subsets of $\sG$.
\end{itemize}
\end{corollary}

\section{Continuity versus measurability}
\label{sec:BorCont}

Recall that it makes no difference whether we require the functions $f_1, f_2, \ldots, f_k$ from the definition of $k$-shapes to be Borel or continuous. This is because the closure of the set of the corresponding matrices in $\Mk$ is the same in both cases, see~\cite[Lemma~3.1]{KuLoSz2019}. Therefore one would expect that in Definition~\ref{def:shape} we can similarly replace continuous fairly distributed functions by arbitrary fairly distributed functions (such functions are automatically non-negative and bounded Borel by definition). In this section we show that this expectation is correct. It may seem that a simple approximation of a Borel function by a continuous one (in an appropriate $L^1$-norm) is enough to prove this fact. But by doing so, we may lose the fair distribution of the function. Therefore we apply Lemma~\ref{lem:shapes-a} to obtain this result.

\begin{proposition}
\label{prop:BC}
For every $\mu \in \sG$ it holds
\[
C(\mu) = \overline{\{ \Phi(f,\mu) \colon f \in \FD \}},
\]
where the closure is taken in the weak topology.
\end{proposition}

\begin{proof}
By Definition~\ref{def:shape} we have $C(\mu) = \overline{\{ \Phi(f,\mu) \colon f \in \FDC \}}$, and so the inclusion $C(\mu) \subseteq \overline{\{ \Phi(f,\mu) \colon f \in \FD \}}$ is trivial.

To prove the opposite inclusion, we only need to show that $\Phi(f,\mu) \in C(\mu)$ for every $f \in \FD$ as the set $C(\mu)$ is closed.
Let $\rho$ be the compatible metric on $\sG$ introduced in the proof of Lemma~\ref{lem:regularity} by~\eqref{eq:metric}.
Let us fix a function $f \in \FD$, and let us also fix $\varepsilon > 0$.
We find $J \in \N$ such that
\[
\sum_{j=J+1}^{\infty} \frac{1}{2^j} < \frac{\varepsilon}{2}.
\]
Then we find $k \in \N$ such that $\left| h_j(x,y) - h_j(x',y') \right| < \frac{\varepsilon}{2}$, $j=1,2,\ldots,J$, whenever the points $(x,y), (x',y')$ from $[0,1]^2$ are such that $|x-x'| \le \frac{1}{k}$ and $|y-y'| \le \frac{1}{k}$. We define $C = \sup_{(x,y)\in[0,1]^2}f(x,y)$. Let $I_i$, $i=1,2,\ldots,k$, be the intervals introduced in the proof of Lemma~\ref{lem:shapes-a} by~\eqref{not:k-net}. Let us define functions $f_1,f_2,\ldots,f_k \colon [0,1] \rightarrow \R$ by
\[
f_i(x) = \int_{u \in I_i} f(x,u) \, d\lambda(u), \qquad x \in [0,1], \qquad i=1,2,\ldots,k.
\]
Then $f_1,f_2,\ldots,f_k$ are non-negative Borel functions such that their sum is the constant function $1$ and such that $\int_{x\in[0,1]} f_i(x) \, d\lambda(x) = 1/k$ for every $i=1,2,\ldots,k$. Repeating step by step the computations from equation~\eqref{eq:close2} in Lemma~\ref{lem:shapes-b} we obtain that the corresponding matrix $M$ belonging to the $k$-shape $C(\mu,k)$, and the corresponding measure $\mu_M \in \widetilde{C}(\mu,k)$, satisfy
\[
\mu_M (I_i \times I_j) = M(i,j) = \Phi(f,\mu) (I_i \times I_j), \qquad i,j = 1,2,\ldots,k.
\]
It easily follows by the choice of $k$ that, for every $j = 1,2,\ldots,J$, it holds
\[
\left| \int_{ (x,y) \in [0,1]^2 } h_j(x,y) \, d\mu_M(x,y) - \int_{ (x,y) \in [0,1]^2 } h_j(x,y) \, d\Phi(f,\mu)(x,y) \right| < \frac{\varepsilon}{2}.
\]
Therefore
\begin{equation*}
\begin{split}
& \rho(\mu_M,\Phi(f,\mu)) \\
=& \sum_{j=1}^{\infty} \frac{1}{2^j} \left| \int_{ (x,y) \in [0,1]^2 } h_j(x,y) \, d\mu_M(x,y) - \int_{ (x,y) \in [0,1]^2 } h_j(x,y) \, d\Phi(f,\mu)(x,y) \right| \\
<& \sum_{j=1}^{J} \frac{\varepsilon}{2^{j+1}} + \sum_{j=J+1}^{\infty} \frac{1}{2^j} < \frac{\varepsilon}{2} + \frac{\varepsilon}{2} = \varepsilon.
\end{split}
\end{equation*}
As $\varepsilon > 0$ was chosen arbitrarily, it follows that $\Phi(f,\mu) \in \overline{ \bigcup_{k\in\N} \widetilde{C}(\mu,k) }$. Finally, this means that $\Phi(f,\mu) \in C(\mu)$ by Lemma~\ref{lem:shapes-a}.
\end{proof}

\section{Acknowledgments}

The author would like to express many thanks to Jan Greb\'ik, Jan Hladk\'y, Aranka Hru\v{s}kov\'a and Israel Rocha for fruitful discussions on s-convergence and other related topics. The author would also like to thank the anonymous referees for their valuable comments.

\bibliographystyle{plain}
\bibliography{bibliography}

\end{document}